\documentclass[11pt,reqno]{amsart}
\usepackage{amsmath,amsthm,amsfonts,amssymb,amscd,amsbsy}
\usepackage[utf8]{inputenc}
\usepackage{array,colortbl}
\usepackage{tikz}
\usepackage{stix}
\usetikzlibrary{cd,arrows,babel}
\usepackage{stmaryrd}
\usepackage{mathtools}
\usepackage{graphicx}
\usepackage[english]{babel}
\usepackage{fullpage}
\usepackage{comment}
\usepackage{relsize}

\usepackage{graphicx}
\usepackage[colorlinks=true, allcolors=blue]{hyperref}

\numberwithin{equation}{section}

\usepackage{lineno}

\newtheorem{theorem}{Theorem}[section]
\newtheorem{lemma}[theorem]{Lemma}

\newtheorem{proposition}[theorem]{Proposition}
\newtheorem{corollary}[theorem]{Corollary}
\newtheorem{claim}[theorem]{Claim}

\theoremstyle{definition}
\newtheorem{remark}[theorem]{Remark}
\newtheorem{definition}[theorem]{Definition}

\newtheorem{example}[theorem]{Example}
\newtheorem{question}[theorem]{Question}
\newtheorem{fact}[theorem]{Fact}

\numberwithin{equation}{section}

\newcommand{\N}{\mathbb N}

\begin{document}
\title{The lattice Sch\"affer constant}

\author{Michael A. Rinc\'on-Villamizar}
\address{Universidad Industrial de Santander (UIS), Escuela de Matem\'aticas, Bucaramanga, Colombia}
\email{marinvil@uis.edu.co}

\author[Oikhberg]{Timur Oikhberg} 
\address{Dept. of Mathematics, University of Illinois, Urbana-Champaign}
\email{oikhberg@illinois.edu}

\thanks{The first author was supported by Universidad Industrial de Santander}



\begin{abstract}
For a Banach lattice $X$, its lattice Sch\"affer constant is defined by: 
\begin{gather*}
    \lambda^+(X)=\inf\{\max\{\|x+y\|,\|x-y\|\}\,\colon\,\|x\|=\|y\|=1,x,y\geq{\bf0}\}.
\end{gather*}
In this paper, we investigate this constant, as well as the companion parameter 
\begin{gather*}
    \beta(X)=\inf\{\|x\vee y\|\,\colon\,\mbox{$\|x\|=\|y\|=1$, $x,y\geq{\bf0}$ and $x\wedge y={\bf0}$}\}.
\end{gather*}
Our main results fall into two groups. (1) We link the behavior of the parameters $\lambda^+$ and $\beta$ to the global properties of the lattice $X$. For instance, 
we prove that (i) if $\lambda^+(X)>1$, then the Banach lattice $X$ is a KB-space, and moreover, it satisfies a lower $q$-estimate for some $q\in(1,\infty)$; (ii) $\lambda^+(X)=1$ if and only if $X$ contains lattice-almost isometric copies of $\ell_\infty^2$; and (iii) that $\lambda^+(X)=2$ if and only if  $X$ is an abstract $L$-space. 
(2) We establish inequalities relating $\lambda^+(X)$ to the characteristics of monotonicity, $\varepsilon_{0,m}(X)$ and $\tilde\varepsilon_{0,m}(X)$.
Along the way, we compute $\lambda^+(X)$ and $\beta(X)$ for various Banach lattices $X$.
\end{abstract}


\maketitle

\section{Introduction}

Many authors have explored the idea of associating to a given Banach space or Banach lattice a real function or a numerical parameter, in order to analyze some specific geometric properties of that space (or lattice). For instance, J. Gao introduced a parameter in \cite{gao}, now known as the James constant, which quantifies the non-squareness of the unit ball. The James constant has been widely studied by numerous researchers (see \cite{kato-maligranda-takahashi} and the references therein for a comprehensive overview).

Another parameter, introduced by J. J. Sch\"affer in \cite{schaffer} and perhaps less well known, is the so-called Sch\"affer constant, defined for a Banach space $X$ as
\begin{equation*}
    \lambda(X)=\inf\{\max\{\|x+y\|,\|x-y\|\}\,\colon\,\|x\|=\|y\|=1\}.
\end{equation*}
This constant is related to the concept of uniform non-squareness introduced by Sch\"affer in 1970. As K. Jarosz demonstrated in \cite{jarosz}, the Sch\"affer constant has applications in the study of vector-valued extensions of the Banach–Stone theorem. Furthermore, in the pursuit of vector-valued lattice extensions of the Banach–Stone theorem, E. M. Galego and the first author introduced a lattice version of the Sch\"affer constant in \cite{galego-me positive} for a real Banach lattice $X$, namely:
\begin{equation*}
    \lambda^+(X)=\inf\{\max\{\|x+y\|,\|x-y\|\}\,\colon\,\mbox{$x,y\geq{\bf0}$ and $\|x\|=\|y\|=1$}\}.
\end{equation*}
We refer to this parameter here as the \textbf{lattice Sch\"affer constant}. Notice that $\lambda(X)\leq\lambda^+(X)$ for every Banach lattice $X$. Moreover, the equality does not hold in general (see examples below).

The aim of this paper is to study the lattice Sch\"affer constant, as well as its companion parameter: 
\begin{gather*}
    \beta(X)=\inf\{\|x\vee y\|\,\colon\,\mbox{$\|x\|=\|y\|=1$, $x,y\geq{\bf0}$ and $x\wedge y={\bf0}$}\}.
\end{gather*}

We begin by outlining the organization of the paper. In Section \ref{preliminaries}, we review key facts about Banach lattices, and recall the definitions of the  lower and upper moduli of uniform monotonicity of a Banach lattice 
$X$, denoted by $\delta_{m,X}(\cdot)$ and $\sigma_X(\cdot)$, respectively. We also gather some key facts about these moduli and their corresponding characteristics, $\varepsilon_{0,m}(X)$ and $\tilde\varepsilon_{0,m}(X)$.

In Section \ref{otros parametros}, we examine the relationship between the lattice Sch\"affer constant and the Riesz angle of a Banach lattice, $\alpha(X)$, as introduced by J.M. Borwein and B. Sims in \cite{borwein-sims}. Additionally, we compute the value $\lambda^+(X)$ and $\beta(X)$ for several Banach lattices $X$. 

In Section \ref{stability} we show that the parameters $\lambda^+$ and $\beta$ are preserved when passing to the second dual and to the ultrapower. Further, these quantities can be computed using finite dimensional sublattices only.

Section \ref{global} explores some connections between our parameters and the ``global'' properties of Banach lattices, such as their convexity, concavity, or being a KB-space.

Finally, in Section \ref{algunas propiedades de lamba^+} we connect the lattice Sch\"affer constant to the moduli of monotonicity.
 For instance, we prove that for every Banach lattice $X$, the inequality $\lambda^+(X)\leq2-\tilde\varepsilon_{0,m}(X)$ holds. Also we show that $\varepsilon_{0,m}(X)<1$ if and only if $\lambda^+(X)>1$. As a consequence, we prove that if $\tilde\varepsilon_{0,m}(X)+\alpha(X)<2$, then $X$ is superreflexive. 

\section{Preliminaries}\label{preliminaries}

The standard terminology and notation of Banach space and Banach lattice theory will be used and can be found in \cite{AK} and \cite{meyer}. The Banach spaces considered here are assumed to be real. 
If $X$ is a Banach lattice, $X^+$ denotes the positive cone of X. If $X$ is a Banach space, $S_X$ and $B_X$ stand for the unit sphere and the closed unit ball, respectively. Also, we set $S_X^+=S_X\cap X^+$ and $B_X^+=B_X\cap X^+$. If $X$ and $Y$ are Banach lattices, a bijective operator $T\colon X\to Y$ is called \textit{Banach lattice isomorphism} if 
$T(x\vee y)=Tx\vee Ty$ for all $x,y\in X$. A \textit{Banach lattice isometry} is a Banach lattice isomorphism which is also an isometry.

We briefly recall the definition of the ultrapower of a Banach space (see e.g.~\cite{Heinrich} for more information). Let $X$
be a Banach space and let $\mathcal U$ be a non-trivial ultrafilter on $\mathbb N$. If $(a_n)$ is a bounded sequence of real numbers, it can be shown that there is an unique real number $a$ such that
for each $\varepsilon>0$, the set $\{n\in\mathbb N\,\colon\,|a_n-a|<\varepsilon\}$ belongs to $\mathcal U$. The number $a$ is called the $\mathcal U$-limit of the sequence $(a_n)$
and it is denoted by $\mathcal U-\lim a_n$.

Let $\ell_\infty(X)$ be the Banach space of all bounded sequences $(x_n)$ in $X$ endowed with the natural norm. The set $c_{0,\mathcal U}(X)=\{(x_i)\in\ell_\infty(X)\,\colon\,\mathcal U-\lim\|x_i\|=0\}$ is a closed subspace of $\ell_\infty(X)$. The \textit{ultrapower of} $X$ \textit{following $\mathcal U$} (or \textit{along $\mathcal U$}) is defined as $X_\mathcal U\coloneqq \ell_\infty(X)\big/c_{0,\mathcal U}(X)$ with the quotient norm. The elements of $X_\mathcal U$ are denoted by $(x_n)_\mathcal U$.
An instructive exercise is to show that the norm of $(x_n)_\mathcal U\in X_\mathcal U$ is given by $\|(x_n)_\mathcal U\|=\mathcal U-\lim\|x_n\|$.

Finally, if $X$ is a Banach lattice, the formulas
\begin{gather*}
    (x_n)_\mathcal U\vee (y_n)_\mathcal U\coloneqq(x_n\vee y_n)_\mathcal U\quad\mbox{and}\quad (x_n)_\mathcal U\wedge (y_n)_\mathcal U\coloneqq(x_n\wedge y_n)_\mathcal U
\end{gather*}
for each $(x_n)_\mathcal U,(y_n)_\mathcal U\in X_\mathcal U$ define a natural Banach lattice structure on the ultrapower $X_\mathcal U$. The following fact is a consequence of definitions.

\begin{fact}\label{embedding into the ultrapower}
If $X$ is a Banach lattice, then the map $x\in X\mapsto(x,x,\ldots)_\mathcal U\in X_\mathcal U$ is a Banach lattice isometry.
\end{fact}


Now we introduce two useful functions. Let $X$ be a Banach lattice. For $\varepsilon\in[0,1]$ we set
\begin{align*}
    \sigma_X(\varepsilon)&=\inf\{\|x+\varepsilon y\|-1\,\colon\,x,y\in S_{X}^+\},\quad\mbox{and}\\
     \delta_{m,X}(\varepsilon)&=\inf\{1-\|x-y\|\,\colon\,{\bf0}\leq y\leq x,\|x\|\leq1,\|y\|\geq\varepsilon\}.
\end{align*}
These functions are called the \textit{upper modulus of monotonicity} and the \textit{lower modulus of uniform monotonicity} of $X$, respectively. It is known that $\delta_{m,X}$ is non-decreasing in $[0,1]$ and continuous in $[0,1)$. Associated with these modules, we have the following numbers called \textit{characteristics of monotonicity} of $X$: 
\begin{align*}
    \varepsilon_{0,m}(X)&=\sup\{\varepsilon\in[0,1)\,\colon\,\delta_{m,X}(\varepsilon)=0\},\quad\mbox{and}\\
    \tilde\varepsilon_{0,m}(X)&=\sup\{\varepsilon\in[0,1)\,\colon\,\sigma_X(\varepsilon)=0\}.
\end{align*}
We say that $X$ is \textit{uniformly monotone} if $\delta_{m,X}(\varepsilon)>0$ for every $\varepsilon\in(0,1]$. It is known that
$X$ is uniformly monotone if and only if $\varepsilon_{0,m}(X)>0$ if and only if $\tilde\varepsilon_{0,m}(X)>0$. This is a consequence of the
inequalities (see \cite[Theorem 1]{hudzik-kaczmarek}):
\begin{gather}\label{inequalities envolving e_0 and tilde e_0}
    \varepsilon_{0,m}(X)\leq\tilde\varepsilon_{0,m}(X)\leq2\varepsilon_{0,m}(X).
\end{gather}
For more information on these objects, see \cite{fora and all}.

For a proof of the following statements, see \cite{kurcII} and \cite[Theorem 2.1]{mar-prus}, respectively.

\begin{fact}\label{Inequalities between delta and eta}
    Let $X$ be a Banach lattice. For all $\varepsilon\in(0,1)$ we have
    \begin{gather*}
        \frac{\delta_{m,X}\left(\dfrac{\varepsilon}{1+\varepsilon}\right)}{1-\delta_{m,X}\left(\dfrac{\varepsilon}{1+\varepsilon}\right)}\leq\sigma_X(\varepsilon)\leq \frac{\delta_{m,X}(\varepsilon)}{1-\delta_{m,X}(\varepsilon)}
    \end{gather*}
\end{fact}

\begin{fact}\label{relationship between delta and eta}
    Let $X$ be a Banach lattice. For all $\varepsilon\in[0,1]$ we have
     \begin{gather*}
    \delta_{m,X}\left(\frac{\varepsilon}{1+\sigma_X(\varepsilon)}\right)=\frac{\sigma_X(\varepsilon)}{1+\sigma_X(\varepsilon)}.
    \end{gather*}
\end{fact}

\begin{lemma}\label{continuity of eta}
    Let $X$ be a Banach lattice. The map $\varepsilon\in[0,1]\mapsto\sigma_X(\varepsilon)\in[0,1]$ is 1-Lipschitz continuous.
\end{lemma}

\begin{proof}
   If $t,s\in[0,1]$, we have
    \begin{gather*}
        |\|x+ty\|-\|x+sy\||\leq|t-s|.
    \end{gather*}
    for all $x,y\in S_X^+$. By taking infima, we conclude that $|\sigma_X(t)-\sigma_X(s)|\leq|t-s|$.
\end{proof}

\begin{corollary}\label{sigma(e_0)=0}
    Let $X$ be a Banach lattice. Then, $\sigma_X(\tilde\varepsilon_{0,m}(X))=0$.
\end{corollary}

\begin{remark}
This work examines quantitative characterizations of monotonicity in a Banach lattice. These are connected with the geometry of its unit ball. We do not study these connections here; an interested reader is referred to e.g.~\cite{HKM}.
\end{remark}

\section{Some parameters associated to Banach spaces and Banach lattices}\label{otros parametros}

In this section, we introduce the lattice Sch\"affer constant and explore its relationship with various other constants associated with Banach spaces and Banach lattices. Additionally, we provide examples of the lattice Sch\"affer constant for specific Banach lattices.

Recall that if $X$ is a Banach space, the James and Sch\"affer constants of $X$ are defined respectively by
\begin{align*}
 J(X)&=\sup\{\inf\{\|x-y\|,\|x+y\|\}\,\colon\,x,y\in S_X\},\quad\mbox{and}\\
     \lambda(X)&=\inf\{\max\{\|x-y\|,\|x+y\|\}\,\colon\,x,y\in S_X\}.
\end{align*}

If $X$ is a Banach lattice, then similar parameters may be defined using its order structure. For instance, \cite{lifsic} and \cite{tsek} defined the \textit{Riesz angle} of a Banach lattice $X$:
    \begin{gather*}
        \alpha(X)=\sup\{\|x\vee y\|\,\colon\,x,y\in B_X^+\}.
\end{gather*}
This parameter was used in \cite{borwein-sims} to study fixed point properties.

In this paper we shall focus on the following parameter, defined in \cite{galego-me}.

\begin{definition}
    Let $X$ be a Banach lattice. The lattice Sch\"affer constant of $X$ is the number
    \begin{gather*}
    \lambda^+(X)=\inf\{\max\{\|x-y\|,\|x+y\|\}\,\colon\,x,y\in S_X^+\}.
\end{gather*}
\end{definition}

Clearly $\lambda^+({\mathbb R}) = 2$ and $\alpha({\mathbb R}) = 1$. We shall be concerned with lattices of dimension at least $2$.
Begin by providing alternative formulas for $\lambda^+(X)$ and $\alpha(X)$ (which also make clear that $\alpha(X)$ is the lattice counterpart of the James constant $J(X)$).


\begin{proposition}\label{formulas for lambda and alpha}
    Let $X$ be a Banach lattice, with $\dim X \geq 2$. Then
    \begin{align}\label{formulas for lambda}
         \lambda^+(X)&=\inf\{\|x+y\|\,\colon\,x,y\in S_X^+\}\nonumber\\
         &=\inf\{\|\,|x|+|y|\,\|\,\colon\,x,y\in S_X\}\\
         &=\inf\{\|\,|x-y|\vee|x+y|\,\|\,\colon\,x,y\in S_X\}\nonumber\\
     &=\inf\{\|x+y\|\,\colon\,x,y\in X^+,\,\min\{\|x\|,\|y\|\}\geq1\}.\nonumber
    \end{align}
    Also, we have
     \begin{align}\label{formulas for alpha}
         \alpha(X)
         &=\sup\{\|x\vee y\|\,\colon\,x,y\in B_X^+, x \wedge y = {\bf0}\}\nonumber\\
         &=\sup\{\|x-y\|\,\colon\,x,y\in S_X^+\}\nonumber\\
         &=\sup\{\|\,|x|-|y|\,\|\,\colon\,x,y\in S_X\}\\
         &=\sup\{\|\,|x-y|\wedge|x+y|\,\|\,\colon\,x,y\in S_X\}\nonumber\\
          &=\sup\{\|\,|x-y|\wedge|x+y|\,\|\,\colon\,x,y\in B_X\}.\nonumber
    \end{align}
\end{proposition}

\begin{proof}
\eqref{formulas for lambda}:  The first identity follows because $|x-y|\leq x+y$ for all $x,y\in X^+$. The second one is consequence from the fact
  that $\|\,|x|\,\|=\|x\|$ for each $x\in X$. Now, since $|x|+|y|=|x-y|\vee|x+y|$ for all $x,y\in X$, we obtain
  the third equality. Finally, if $D(X)=\inf\{\|x+y\|\,\colon\,x,y\in X^+,\,\min\{\|x\|,\|y\|\}\geq1\}$, then 
  $D(X)\leq\lambda^+(X)$. To prove the opposite inequality, if $\varepsilon>0$ is given, there exist $x,y\in X^+$ with 
  $\min\{\|x\|,\|y\|\}\geq1\}$ such that $\|x+y\|<D(X)+\varepsilon$. Notice that
  \begin{gather*}
      \lambda^+(X)\leq\left\|\frac{x}{\|x\|}+\frac{y}{\|y\|}\right\|\leq\|x+y\|<D(X)+\varepsilon.
  \end{gather*}
  From the arbitrariness of $\varepsilon$ we conclude that $\lambda^+(X)\leq D(X)$.

\eqref{formulas for alpha}: Clearly $\alpha(X) \geq \sup\{\|x\vee y\|\,\colon\,x,y\in B_X^+, x \wedge y = {\bf0}\}$. To prove the converse, recall that, by \cite{tsek} (see \cite{wnuk} for a different proof), $\alpha(X) = \alpha(X^{**})$. Being a dual space, $X^{**}$ is $\sigma$-order complete, hence by \cite{wnuk}, for any $\varepsilon > 0$ there exist disjoint $u,v\in B_{X^{**}}^+$ so that $\alpha(X^{**})<\|u\vee v\|+\varepsilon$.
By \cite{bernau}, for such $u,v$ there exist disjoint $x,y \in B_X^+$ so that $\|x\vee y\|+\varepsilon > \|u\vee v\|$. Thus, $\alpha(X) = \alpha(X^{**})<\|x\vee y\|+2\varepsilon$, and we obtain the first formula for $\alpha(X)$ by taking $\varepsilon$ arbitrarily small.

  Since $x+y=|x-y|$ for all $x,y\in X^+$ with $x\wedge y={\bf0}$, the second equation holds. The third equation follows since $||x|-|y||=|x-y|\wedge|x+y|$ for all $x,y\in X$; a similar reasoning can be used for the fourth one.
  To prove the last equation, let $C(X)=\sup\{\|\,|x-y|\wedge|x+y|\,\|\,\colon\,x,y\in B_X\}$. Clearly $\alpha(X)\leq C(X)$. For the reverse inequality, notice that for each $x,y\in X$, the map $t\in\mathbb R\mapsto\|\,|x-ty|\wedge|x+ty|\,\|$ is convex and even, and hence it is non-decreasing. Thus if  and $\varepsilon>0$ is given, we have for some $x,y\in B_X\setminus\{{\bf0}\}$ that 
    \begin{align*}
        C(X)-\varepsilon&<\|\,|x-y|\wedge|x+y|\,\|\\
        &\leq\left\|\,\left|x-\frac{y}{\|y\|}\right|\wedge\left|x+\frac{y}{\|y\|}\right|\,\right\|\\
        &\leq\left\|\,\left|\frac{x}{\|x\|}-\frac{y}{\|y\|}\right|\wedge\left|\frac{x}{\|x\|}+\frac{y}{\|y\|}\right|\,\right\|\\
        &\leq \alpha(X).
    \end{align*}
    Therefore, by taking $\varepsilon\to0^+$ we obtain that $C(X)\leq\alpha(X)$. 
\end{proof}

From the previous proposition, we obtain:

\begin{corollary}\label{inequalities involving lambda, alpha and james}
    Let $X$ be a Banach lattice. Then
    \begin{gather*}
        \lambda(X)\leq\lambda^+(X)\leq\alpha(X)\leq J(X).
    \end{gather*}
    Consequently, if $\lambda(X)=\sqrt{2}$, then $\lambda^+(X)=\alpha(X)=\sqrt{2}$.
\end{corollary}

\begin{proof}
We will check the middle inequality of the centered formula, since other paers are similar.     If $x,y\in X^+$ satisfy $\|x\|=\|y\|=1$ and $x\wedge y={\bf0}$, since $x+y=x\vee y+x\wedge y=x\vee y$, we have
\begin{gather*}
    \lambda^+(X)\leq\|x+y\|=\|x\vee y\|\leq \alpha(X).
\end{gather*}
For the ``consequently'' statement, recall that $J(X)\lambda(X)=2$ \cite[Theorem 2]{kato-maligranda-takahashi}.
\end{proof}

Now in connection with the lattice Sch\"affer constant, we introduce another parameter for the Banach lattice $X$ with $\dim X \geq 2$, namely
 \begin{gather*}
        \beta(X)=\inf\{\|x\vee y\|\,\colon\,x,y\in S_X^+,\, x\wedge y={\bf0}\}=\inf\{\|x+ y\|\,\colon\,x,y\in S_X^+,\, x\wedge y={\bf0}\}.
    \end{gather*}
Clearly, $\lambda^+(X)\leq\beta(X) \leq \alpha(X)$ for every Banach lattice $X$. By Corollary \ref{inequalities involving lambda, alpha and james}, $\lambda^+(X) = \beta(X) = \alpha(X)$ when $\lambda(X)=\sqrt{2}$. The equality $\lambda^+(X) = \beta(X)$ also holds in the two-dimensional case:

\begin{proposition}\label{2-dim}
If $X$ is a two-dimensional Banach lattice, then $\beta(X) = \lambda^+(X)$.
\end{proposition}

\begin{proof}
We can and do identify $X$ with $\mathbb R^2$ (see \cite[Section II.3]{sch74}), and assume the normalization $\|(1,0)\| = \|(0,1)\| = 1$. Clearly $\beta(X) = \|(1,1)\|$. We shall show that $\lambda^+(X) \geq \|(1,1)\|$; as we already know that $\lambda^+(X) \leq \beta(X)$, our goal will then be reached.

Suppose, for the sake of contradiction, that $\beta(X) > \lambda^+(X)$. By replacing $B_X$ with $B_X + \varepsilon B_{\ell_2^2}$, and scaling to preserve normalization, we can assume that $X$ is smooth, that is, for every $x \in X\backslash\{0\}$ there exists a unique $f \in S_{X^*}$ with $f(x) = \|x\|$; write $f_x = f$. It is easy to observe that, if $x$ is positive, then so is $f_x$. By e.g. \cite[Chapter 2]{Diestel}, for any $y \in X$, $\|x+ty\| = \|x\| + t f_x(y) + o(t)$ as $t \to 0$. 

By compactness, we can find $x,y \in S_X^+$ so that $\|x+y\| = \lambda^+(X)$. As $x$ and $y$ are not disjoint, at least one of them, say $x$, has full support.

Notice that $f_x = f_z$, where $z = x+y$. Indeed, otherwise, pick $u$ with $f_x(u) = 0$ and $f_z(u) < 0$. For $t$ small, $x+tu$ is still positive, and $\|x+tu\| = 1 + o(t)$. Take $x'_t = (x+tu)/\|x+tu\|$, then $x'_t = 1 + tu + o(t)$, and $\|x'_t + y\| = 1 + t f_z(u) + o(t)$, which is less than $\|x+y\| = \lambda^+(X)$ for $t$ small, which cannot happen. 

Thus, $f_x = f_z$. If $y$ has full support, we similarly conclude that $f_y = f_z$, hence $\lambda^+(X) = f_z(x+y) = \|x\| + \|y\| = 2$. By Theorem \ref{characterization of banach lattices having l=2} below, $X = \ell_1^2$, and $\lambda^+(X) = 2 = \|(1,1)\| = \beta(X)$.

Now suppose $y$ does not have full support, say $x = (a,b)$ and $y = (1,0)$, with $a,b > 0$. As noted above, for $x = (a,b)$ and $z = (a+1,b)$ we have $f_x = f_z = f = (u,v)$. We shall show that $\|(1,1)\| \leq \|(a+1,b)\|$, which then implies $\beta(X) = \|(1,1)\| \leq \lambda^+(X)$.

Consider first the case of $a \leq b$. Observe that $f$ norms both $(a,b)$ and $(a+1,b)$, hence also $(a',b') = (a+1,b)/\|(a+1,b)\|$. For $t \in [0,1]$,
$$
1 \geq \|(1-t) (a,b) + t(a',b')\| \geq f \big( (1-t) (a,b) + t(a',b') \big) = 1 ,
$$
hence all the ``$\geq$'' signs above are in fact equalities. As $a' \geq b'$, one find $t$ so that $(1-t) (a,b) + t(a',b')$ is a multiple of $(1,1)$, hence $f$ also norms $(1,1)$. That is, $\|(1,1)\| = f((1,1)) = u+v$. Clearly $v \leq 1$, and $1 = \|x\| = f(x) = au+bv$, hence
$$
\|(1,1)\| = f \big((1,1)\big) = u+v \leq 1+u = (au+bv) + u = f \big((a+1,b)\big) = \|(a+1,b)\| .
$$

Now suppose $a > b$. Let $w = (1-b)/a$ (note that $a \vee b \leq 1 \leq a+b$, hence $0 \leq w \leq 1$), and consider $g = (u,1) \in X^*$. Then $g((a,b)) = 1 = g((0,1))$.

Note that $(0,1), (a,b) \in S_X$. Denote by $L$ the straight line passing through these two points, then $g|_L = 1$. Further, denote by $S$ the segment connecting $(0,1)$ with $(a,b)$. By convexity, $S \subset B_X$, while $L \backslash S$ does not meet the interior of $B_X$.

As $a < b$, $S$ contains $(1,1)/t$, for some $t > 1$. Then $\|(1,1)\| \leq t$. On the other hand, $1 = g((1,1)/t)$, so $t = g((1,1)) = w+1$, hence $\|(1,1)\| \leq w+1$.

Now find $s>1$ so that $(a+1,b)/s \in L \backslash S$. Then $\|(a+1,b)\| \geq s$. However,
$1 = g((a+1,b))/s$, and so, $s = g((a+1,b)) = (a+1)w + b = (aw+b) + w = 1 + w$. Thus, $\|(a+1,b)\| \geq 1+w \geq \|(1,1)\|$.
\end{proof}

Example \ref{counterexample} and Remark \ref{higher dim counterexample} show that Proposition \ref{2-dim} fails in dimensions higher than $2$.

\begin{example}\label{counterexample}
    There exists a $3$-dimensional Banach lattice $X$ for which $\lambda^+(X) < \beta(X)$. 

   Let $X=\mathbb R^3$ be endowed with the lattice norm
$$
\|(x,y,z)\| = \Big( |x| + \frac{|z|}2 \Big) \bigvee \Big( |y| + \frac{|z|}2 \Big) \bigvee \Big( \frac23 |x| + \frac23 |y| + \frac{|z|}3 \Big) \bigvee \frac56 \big( |x| + |y| \big) ,\quad (x,y,z)\in X.
$$
We show that
$$
\beta(X) = \frac{15}{11} > \frac43 \geq \lambda^+(X) .
$$
First we estimate $\lambda^+(X)$. Take norm one vectors
$$
u = \Big( \frac45, 0, \frac25 \Big)  \quad\mbox{and} \quad v = \Big( 0 , \frac45, \frac25 \Big) .
$$
Then $\|u+v\| = 4/3 \geq \lambda^+(X)$. Next we compute $\beta(X)$, by finding an upper estimate for $\|u+v\|$ when $u, v\in S_X^+$ and $u\wedge v={\bf0}$. By symmetry, it suffices to consider two cases: (i) $u = (a,b,0)$, $v = (0,0,2)$, and (ii) $u = (a,0,c)$, $v = (0,1,0)$.

Case (i): We have $a+b = 6/5$. Hence,
$$
\|u+v\| = \|(a,b,2)\| \geq \frac23(a+b) + \frac23 = \frac{22}{15} > \frac{15}{11} > \frac43 .
$$

Case (ii): Observe that $a+c/2 = 1$. Therefore,
$$
\|u+v\| = \|(a,1,c)\| \geq \Big( 1 + \frac{c}2 \Big) \bigvee \frac56 (a+1) .
$$
The minimum of the right hand side under the constraint $a+c/2 = 1$ (that is, $c/2 = 1-a$) occurs when
$$
\frac56 (a+1) = 1 + \frac{c}2 = 2 - a .
$$
Then $a = 7/11$, and we obtain $\|u+v\| = 15/11$. 
\end{example}

\begin{remark}\label{higher dim counterexample}
Suppose $X$ is as in Example \ref{counterexample}. Let $W = L_1(E,\Sigma,\mu)$ for some $\sigma$-finite measure space $(E,\Sigma,\mu)$, and let $Z = X \oplus_1 W$. Then $\lambda^+(Z) < \beta(Z)$.

To establish this, it suffices to show that $\lambda^+(Z) =\lambda^+(X)$ and $\beta(Z) = \beta(X)$.
To this end, pick $x, y \in S_Z^+$. Write $x = (x_1, x_2)$, with $x_1 \in X, x_2 \in W$. Similarly, we have $y = (y_1,y_2)$. Note that $\|x \vee y\| = \|x_1 \vee y_1\| + \|x_2\| + \|y_2\|$. 

Consider $x' = x_1/\|x_1\|, y' = y_1/\|y_1\| \in S_X^+$. Then 
$$
\|x' \vee y'\| \leq \|x_1 \vee y_1\| + \|x' - x_1\| + \|y' - y_1\| = \|x_1 \vee y_1\| + \|x_2\| + \|y_2\| \leq \|x \vee y\| .
$$
If, in addition, $x$ and $y$ are disjoint, then so are $x'$ and $y'$. Thus, in the definition of $\lambda^+(Z)$ and $\beta(Z)$, we can restrict ourselves to taking the infimum over appropriate $x', y' \in X$.
\end{remark}

\begin{remark}
   Clearly $1 \leq \lambda^+(X) \leq \beta(X) \leq 2$ for any Banach lattice $X$, and Example \ref{counterexample} shows that the middle inequality can be strict. We do not know the maximal values of $\beta(X)/\lambda^+(X)$ or $\beta(X) - \lambda^+(X)$. 
\end{remark}

Now we show examples of $\lambda^+(X)$ for some particular Banach lattices $X$.

\begin{example}\label{example L_p}
    Let $p\in[1,\infty)$ and $(E,\Sigma,\mu)$ be a measure space.
    If $X=L_p(E,\Sigma,\mu)$, we have $\lambda^+(X)=\beta(X)=2^{1/p}$. Indeed, if $f,g\in S_X^+$, from
    the elementary inequality $(u+v)^p\geq u^p+v^p$ for all $u,v\geq0$ we obtain
    \begin{gather*}
        \|f+g\|^p=\int_X[f(t)+g(t)]^p\,d\mu\geq\int_Xf(t)^p\,d\mu+\int_Xg(t)^p\,d\mu=2.
    \end{gather*}
    Thus, $\lambda^+(X)\geq2^{1/p}$. Now, if $f\wedge g={\bf0}$, then $\|f+g\|^p=\|f\|^p+\|g\|^p=2$. 
    Whence, $\beta(X) \leq 2^{1/p}$. To summarize, $2^{1/p} \leq \lambda^+(X) \leq \beta(X) \leq 2^{1/p}$, yielding the desired result.
\end{example}

\begin{remark}
    Notice that if $1\leq p<2$ and $X=(\ell_p,\|\cdot\|_p)$, then $\lambda(X)=2^{1-1/p}<2^{1/p}=\lambda^+(X)$.
\end{remark}

\begin{example}\label{lambda^+(X)=sqrt{2} and X is not hilbert}
    Let $X=\mathbb R^2$ be endowed with the lattice norm
    \begin{gather*}
        \|(x,y)\|=\max\{\|(x,y)\|_\infty,\frac{1}{\sqrt{2}}\|(x,y)\|_1\}, \quad (x,y)\in X.
    \end{gather*}
    Since $\lambda(X)=\sqrt{2}$, from Corollary \ref{inequalities involving lambda, alpha and james} we obtain $\lambda^+(X)=\sqrt{2}$.
\end{example}

\begin{example}
    Let $1\leq a<\sqrt{2}$ and let $X=\mathbb R^2$ with the lattice norm
    \begin{gather*}
        \|x\|_a=\max\{\|x\|_2,a\|x\|_\infty\}, \quad x\in X.
    \end{gather*}
    From \cite[Theorem 2 and Example 5]{kato-maligranda-takahashi} we obtain $\lambda^+(X)\geq\lambda(X)=\sqrt{2}/a$.
    Moreover, if $x=(1/a,0)$ and  $y=(0,1/a)$, we have $x,y\in S_X^+$, $x\wedge y={\bf0}$ and $\|x+y\|=\sqrt{2}/a$. Therefore
    $\beta(X)\leq\sqrt{2}/a$. Thus, $\lambda^+(X)=\beta(X)=\sqrt{2}/a$. 
\end{example}

\section{Stability results}\label{stability}
In this section we show that $\beta$ and $\lambda^+$ are stable under small perturbations, as well as under passage to the second dual or to the ultrapower of a lattice.

To state our first result, recall that a Banach lattice $X$ has the Principal Projection Property (PPP for short) if for every $x\in X$, the principal band generated by $x$ is a projection band (see \cite[p. 38]{A-B} for facts concerning such lattices).

\begin{proposition}\label{+schaefer is determinated by finite dimensional subspaces}
    Let $X$ be a Banach lattice. If $X$ has the PPP, then
\begin{align}
    \lambda^+(X)&=\inf\{\lambda^+(Y)\,\colon\,\mbox{$Y$ is a finite dimensional sublattice of $X$}\}\label{eq:fin dim subsp},\quad\mbox{and}\\
    \beta(X)&=\inf\{\beta(Y)\,\colon\,\mbox{$Y$ is a finite dimensional sublattice of $X$}\}\label{eq:fin dim subsp beta}.
\end{align}
\end{proposition}

    \begin{proof}
We only show \eqref{eq:fin dim subsp}, as \eqref{eq:fin dim subsp beta} is handled similarly.

Clearly, $\lambda^+(X)\leq\inf\{\lambda^+(Y)\,\colon\,\mbox{$Y$ is finite dimensional sublattice of $X$}\}.$ To prove the converse, fix $A > \lambda^+(X)$. Let $x,y \in S_X^+$ be such that $\|x+y\| < A$.
 By \cite[Theorem 2.8 and the discussion preceding it]{A-B}, for any $\varepsilon > 0$ we can find disjoint components of $z=x+y$, say $z_1, \ldots, z_N$, so that there exist $x', y' \in S_X^+ \cap Y$, where $Y = \mathrm{span}\,[z_1, \ldots, z_N]$, with $\|x - x'\|, \|y - y'\| < \varepsilon$.
 Then $\lambda^+(Y)\leq\|x'+y'\|<\|z\| + 2\varepsilon < A + 2\varepsilon$. As $A > \lambda^+(X)$ and $\varepsilon > 0$ are arbitrary, the conclusion $\inf\{\lambda^+(Y)\,\colon\,\mbox{$Y$ is finite dimensional sublattice of $X$}\}\leq\lambda^+(X)$ follows.
\end{proof}

The following Lemma improves the inequality proved in \cite[Lemma 2.1]{galego-me}.

\begin{proposition}\label{+schaefer under isomorphisms}
    Let $X$ and $Y$ be Banach lattices and $T\colon X\to Y$ be a Banach lattice isomorphism. Then
    \begin{gather*}
       \frac{1}{\|T\|\|T^{-1}\|}\lambda^+(X)\leq \lambda^+(Y)\leq\|T\|\|T^{-1}\|\lambda^+(X) \quad\mbox{and}\quad 
        \frac{1}{\|T\|\|T^{-1}\|}\beta(X)\leq \beta(Y)\leq\|T\|\|T^{-1}\|\beta(X) .
    \end{gather*}
\end{proposition}

\begin{proof}
Again, we only deal with the inequality for $\lambda^+$, as the case of $\beta$ requires only minor adjustments.

    Let $x,y\in S_X^+$ be given. Notice that
    \begin{gather*}
        \frac{\|T\|}{\|Tx\|}x+\frac{\|T\|}{\|Tx\|}y\leq\max\left\{\frac{\|T\|}{\|Tx\|},\frac{\|T\|}{\|Ty\|}\right\}(x+y)\leq\|T\|\|T^{-1}\|(x+y).
    \end{gather*}
    Thus,
    \begin{gather*}
        \lambda^+(Y)\leq\left\|\frac{Tx}{\|Tx\|}+\frac{Ty}{\|Ty\|}\right\|\leq\left\|\frac{\|T\|}{\|Tx\|}x+\frac{\|T\|}{\|Tx\|}y\right\|\leq\|T\|\|T^{-1}\|\|x+y\|.
    \end{gather*}
    So, $\lambda^+(Y)\leq\|T\|\|T^{-1}\|\lambda^+(X).$
\end{proof}

To go on, we say that $Y$ contains a lattice copy of $X$ if $X$ is Banach-lattice isomorphic to a sublattice of $Y$;
 we also say that $Y$ contains lattice-almost isometric copies of $X$ if for each $\varepsilon>0$, there exist
 a sublattice $Z_\varepsilon$ of $Y$ and a Banach-lattice isomorphism $T_\varepsilon\colon X\to Z_\varepsilon$ such that $\|T_\varepsilon\|\|T_\varepsilon^{-1}\|\leq1+\varepsilon.$ The next corollary is an inmediate consequence of Proposition \ref{+schaefer under isomorphisms}.

\begin{corollary}\label{+schaefer and subspaces}
Let $X$ and $Y$ be Banach lattices.  If $Y$ contains lattice-almost isometric copies of $X$, then $\lambda^+(Y)\leq \lambda^+(X)$ and $\beta(Y)\leq \beta(X)$.
\end{corollary}

If $X$ and $Y$ are Banach lattices, recall that $X$ is said to be \textit{lattice finitely representable} in $Y$ is for every finite
dimensional sublattice $E$ of $X$ and every $\varepsilon>0$, there is a sublattice $F$ of $Y$ and a Banach lattice isomorphism $T\colon E\to F$
such that $\|T\|\|T^{-1}\|\leq1+\varepsilon$. The following fact, established in \cite{bernau}, will be used in the proof of the upcoming theorem.

\begin{fact}\label{finite representability of X** in X}
    For every Banach lattice $X$, the bidual $X^{**}$ is lattice finitely representable in $X$.
\end{fact}
 
\begin{theorem}\label{bidual etc}
 Let $\mathcal U$ be a non-trivial ultrafilter on $\mathbb N$ and $X$ be a Banach lattice. Then
\begin{gather*}
\lambda^+(X)=\lambda^+(X^{**})=\lambda^+(X_\mathcal U) \quad \mbox{and} \quad \beta(X)=\beta(X^{**})=\beta(X_\mathcal U).
\end{gather*}
\end{theorem}

\begin{proof}
Here as well, we only prove the statements involving $\lambda^+$. Begin from the first equality. Since $J_X\colon X\to X^{**}$ is an into Banach lattice isometry, we have $\lambda^+(X^{**})\leq \lambda^+(X)$. On the other hand, 
 by \cite[Theorem 1.47]{A-B}, any $\sigma$-order complete Banach lattice has the PPP, hence any dual Banach lattice must have this property. Consequently, by Proposition \ref{+schaefer is determinated by finite dimensional subspaces} we have
 \begin{gather*}
     \lambda^+(X^{**})=\inf\{\lambda^+(Y)\,\colon\,\mbox{$Y$ is finite dimensional sublattice of $X^{**}$}\}.
 \end{gather*}
Let $\varepsilon>0$ be given and $Y$ be a finite dimensional sublattice of $X^{**}$ such that $\lambda^+(Y)<\lambda^+(X^{**})+\varepsilon$. By Fact \ref{finite representability of X** in X} there is
    a sublattice $Y_0$ of $X$ and a Banach lattice isomorphism $T\colon Y\to Y_0$ such that $\|T\|\|T^{-1}\|\leq 1+\varepsilon$. From Proposition
    \ref{+schaefer under isomorphisms} it follows that 
    \begin{equation*}
    \lambda^+(X)\leq \lambda^+(Y_0)\leq(1+\varepsilon)\lambda^+(Y)<(1+\varepsilon)(\lambda^+(X^{**})+\varepsilon).
    \end{equation*}
    By the arbitrariness of $\varepsilon>0$ we conclude that $\lambda^+(X)\leq \lambda^+(X^{**}).$

    Now we prove that $\lambda^+(X)=\lambda^+(X_\mathcal U).$ From Fact \ref{embedding into the ultrapower} and Corollary \ref{+schaefer and subspaces} we have
    $\lambda^+(X_{\mathcal U})\leq\lambda^+(X)$. Now let $\varepsilon>0$ be given. Thus there are two elements $(x_n)_\mathcal U,(y_n)_\mathcal U\in X_\mathcal U$
    with $\|(x_n)_\mathcal U\|=\|(y_n)_{\mathcal U}\|=1$ such that $\|(|x_n|)_\mathcal U+(|y_n|)_\mathcal U\|=\mathcal U-\lim\|\,|x_n|+|y_n|\,\|<\lambda^+(X_{\mathcal U})+\varepsilon$.
    Let $A\in\mathcal U$ be such that $\|\,|x_n|+|y_n|\,\|<\lambda^+(X_{\mathcal U})+\varepsilon$ for all $n\in\mathbb N$. Also let $B\in\mathcal U$ be such that
    $1-\varepsilon<\min\{\|x_n\|,\|y_n\|\}$ for each $n\in B$. If $m\in A\cap B$, then
    \begin{gather*}
       \lambda^+(X)\leq \left\|\frac{|x_m|}{\|x_m\|}+\frac{|y_m|}{\|y_m\|}\right\|\leq\left\|\frac{|x_m|}{1-\varepsilon}+\frac{|y_m|}{1-\varepsilon}\right\|<\frac{\lambda^+(X_{\mathcal U})+\varepsilon}{1-\varepsilon}.
    \end{gather*}
    Since $\varepsilon>0$ was arbitrary, we obtain $\lambda^+(X)\leq \lambda^+(X_\mathcal U).$
\end{proof}

\begin{remark}
    The proof of Theorem \ref{bidual etc} shows that \eqref{eq:fin dim subsp} and \eqref{eq:fin dim subsp beta} hold for any Banach lattice $X$ (no need to assume the PPP).
\end{remark}

\section{\texorpdfstring{$\alpha$, $\beta$, and $\lambda^+$}{alpha, beta, and lambda+} describe ``global'' properties of Banach lattices}\label{global}

Even though the parameters $\alpha$, $\beta$, and $\lambda^+$ are calculated ``locally'', they can be used to describe ``global'' properties of Banach lattices.
To present our first result in this direction, recall that a Banach lattice $X$ is an abstract $L_p$-space, $1\leq p<\infty$, if  $\|x+y\|=\|x\|+\|y\|$ for each $x,y\in X$ disjoint
(see \cite[Definition 1, p. 131]{lacey}). If $p=1$, $X$ is said to be an abstract $L$-space.

\begin{theorem}\label{characterization of banach lattices having l=2}
Let $X$ be a Banach lattice. The following statements are equivalent:
\begin{enumerate}
    \item[(i)] $\lambda^+(X)=2$.
    \item[(ii)] $\beta(X)=2$.
    \item[(iii)] $X$ is an abstract $L$-space. 
\end{enumerate}
\end{theorem}

\begin{proof}
Since $\lambda^+(X)\leq\beta(X)\leq2$, it follows that (i) implies (ii). Suppose that (ii) holds. 

\begin{claim}\label{ejercicio clásico}
If $x,y\in X$ and $\|x+y\|=\|x\|+\|y\|$, then $\|\alpha x+\beta y\|=\alpha\|x\|+\beta\|y\|$ for $\alpha,\beta\geq0$.
\end{claim}

We may assume that $\beta\geq\alpha$. Then
\begin{gather*}
    \|\alpha x+\beta y\|=\|(\alpha-\beta)x+\beta(x+y)\|\geq\beta(\|x\|+\|y\|)+(\alpha-\beta)\|x\|=\alpha\|x\|+\beta\|y\|,
\end{gather*}
and the claim is proved.

Let $x,y\in X^+\setminus\{{\bf0}\}$ be such that $x\wedge y={\bf0}$. The definition of $\beta(X)$ gives that $\displaystyle\left\|\frac{x}{\|x\|}+\frac{y}{\|y\|}\right\|=2$. By Claim \ref{ejercicio clásico} we have $\left\|a\dfrac{x}{\|x\|}+b\dfrac{y}{\|y\|}\right\|=a+b$ for all $a,b\geq0$. In particular, if $a=\|x\|$ and $b=\|y\|$, we obtain $\|x+y\|=\|x\|+\|y\|$ as required. Therefore, we prove (iii).

Finally, if (iii) is valid, then by \cite[Theorem 1, p. 133]{lacey} we have $\|x+y\|=\|x\|+\|y\|$ for all $x,y\in X^+$. Thus, $\lambda^+(X)=2$.
\end{proof}

    Example \ref{example L_p} implies that $\lambda^+(X)=\beta(X)=2^{1/p}$ when $X$ is an abstract $L_p$-space. The converse is false, per Example \ref{lambda^+(X)=sqrt{2} and X is not hilbert}. So, Theorem \ref{characterization of banach lattices having l=2} cannot be modified to characterize abstract $L_p$-spaces. 
    However, a weaker result (showing that the parameters $\lambda^+$ and $\beta$ can describe global properties of Banach lattices) still holds.
To state the upcoming result, we recall (see e.g. \cite[Section 1.f]{lindenstrauss-tzafriri}) that a Banach lattice $X$ satisfies a lower (resp.~upper) $q$-estimate with constant $M$ ($1\leq q\leq \infty$) if the inequality
\begin{gather*}
    M\left\|\sum_{i=1}^nx_i\right\|\geq \left(\sum_{i=1}^n\|x_i\|^q\right)^{1/q} \quad (\mbox{respectively,}\, \left\|\sum_{i=1}^nx_i\right\|\leq M\left(\sum_{i=1}^n\|x_i\|^q\right)^{1/q}),
\end{gather*}
holds for every family of pairwise disjoint elements $x_1,\ldots,x_n\in X$. For a Banach lattice $X$, we denote by $s(X)$ the infimum of all $q \in [1,\infty]$ for which $X$ has lower $q$-estimate (with some constant); $S(X)$ stands for the supremum of all $q$ for which an upper $q$-estimate holds.
    

\begin{proposition}\label{local vs global}
    If $X$ is an infinite dimensional Banach lattice, then
    \begin{equation}
    2^{1/s(X)} = \lim_n \inf_{Y \subset X, \dim Y = n} \sup_{Z \subset Y} \lambda^+(Z) = \lim_n \inf_{Y \subset X, \dim Y = n} \sup_{Z \subset Y} \beta(Z) .
    \label{eq:inf sup}
    \end{equation}
    Here, the supremum is taken over all sublattices $Z \subset Y$ with $\dim Z \geq 2$, and the infimum -- over all $n$-dimensional sublattices $Y \subset X$. Similarly,
    \begin{equation}
    2^{1/S(X)} = \lim_n \sup_{Y \subset X, \dim Y = n} \inf_{Z \subset Y} \alpha(Z) .
    \label{eq:convex sup inf}
    \end{equation}
\end{proposition}

\begin{proof}
We shall establish \eqref{eq:inf sup}, as \eqref{eq:convex sup inf} is handled similarly.

By the local reflexivity results of \cite{bernau} (for the duality of upper and lower estimates, see also \cite[Proposition 1.f.5]{lindenstrauss-tzafriri}), both the left and the right hand sides remain the same if we can replace $X$ by $X^{**}$. We therefore assume that $X$ is order complete.
Then for every $n$-dimensional sublattice $W \subset X$ there exists a sublattice $Y$ of dimension $n+1$ so that $W \subset Y \subset X$, hence
$$
\inf_{Y \subset X, \dim Y = {n+1}} \sup_{Z \subset Y} \lambda^+(Z) \geq \inf_{W \subset X, \dim W = n} \sup_{Z \subset W} \lambda^+(Z) , 
$$
and similarly with $\beta$. Thus, the sequences $\big(  \inf_{Y \subset X, \dim Y = n} \sup_{Z \subset Y} \lambda^+(Z) \big)_n$ and $\big( \inf_{Y \subset X, \dim Y = n} \sup_{Z \subset Y} \beta(Z) \big)_n$ are increasing in $n$, hence convergent.

    We use a version of Krivine's Theorem \cite{Schep}: if $W$ is an infinite dimensional Banach lattice, then for every $k \in \N$ and $\varepsilon > 0$, there exist disjoint positive norm one $w_1, \ldots, w_k \in W$ so that, for any scalars $\alpha_1, \ldots, \alpha_k$,
    $$ (1+\varepsilon)^{-1} \big( \sum_i |\alpha_i|^{s(W)} \big)^{1/s(W)} \leq \big\| \sum_i \alpha_i w_i \big\| \leq (1+\varepsilon) \big( \sum_i |\alpha_i|^{s(W)} \big)^{1/s(W)} . $$
    For $W$ as above, Example \ref{example L_p} gives $\lambda^+(W) \leq \beta(W) \leq (1+\varepsilon)^2 2^{1/s(W)}$. In \eqref{eq:inf sup}, we can take $Y = Z$ to be the $k$-dimensional $W$ as above. Thus,
    $$ 2^{1/s(X)} \geq \lim_n \inf_{Y \subset X, \dim Y = n} \sup_{Z \subset Y} \lambda^+(Z) \quad\mbox{and}\quad  2^{1/s(X)} \geq\lim_n \inf_{Y \subset X, \dim Y = n} \sup_{Z \subset Y} \beta(Z) . $$

    To prove the converse, we use an ``ultraproduct trick'' (see e.g.~\cite[Theorem 3.6]{Rosenthal}): for every $k \in \N$ and $\varepsilon > 0$ there exists $N = N(k,\varepsilon)$ so that, for any $n \geq N$, and any norm one disjoint $x_1, \ldots, x_n$ there exist disjoint norm one $z_1, \ldots, z_k$ in the span of $x_1, \ldots x_n$, so that, for some $p \in [1,\infty]$, the inequality
    $$ (1+\varepsilon)^{-1} \big( \sum_i |\alpha_i|^p \big)^{1/p} \leq \big\| \sum_i \alpha_i z_i \big\| \leq (1+\varepsilon) \big( \sum_i |\alpha_i|^p \big)^{1/p} $$
    holds for all scalars $\alpha_1, \ldots, \alpha_k$.
    In particular, if $Y$ is an $n$-dimensional sublattice of $X$, it contains a further sublattice $Z$ spanned by $z_1, \ldots, z_k$ as above; then
    $$ \sup_{Z \subset Y} \lambda^+(Z)\geq (1+\varepsilon)^{-2} 2^{1/p}\quad \mbox{and}\quad \sup_{Z \subset Y} \beta(Z) \geq (1+\varepsilon)^{-2} 2^{1/p} . $$
Fix $q > s(X)$, then $X$ has lower $q$-estimate with the constant $C_q$. Using the definition of lower $q$-estimate on $z_1, \ldots, z_k$, we obtain $k^{1/q} \leq C_q (1+\varepsilon) k^{1/p}$. Consequently,
$$ \frac1p \geq \frac1q - \frac{\ln((1+\varepsilon)C_q)}{\ln k} , $$
and so,
$$\sup_{Z \subset Y} \lambda^+(Z), \sup_{Z \subset Y} \beta(Z) \geq (1+\varepsilon)^{-2} 2^{1/q} 2^{-\ln((1+\varepsilon)C_q)/\ln k} .$$
Taking $n$ to be large enough, we can make $\varepsilon$ and $k$ to be as small as possible and as large as possible, respectively. Thus,
$$
2^{1/q} \leq \lim_n \inf_{Y \subset X, \dim Y = n} \sup_{Z \subset Y} \lambda^+(Z) \quad\mbox{and}\quad 2^{1/q} \leq  \lim_n \inf_{Y \subset X, \dim Y = n} \sup_{Z \subset Y} \beta(Z) .
$$
As $q > s(X)$ is arbitrary, we conclude that
$$
2^{1/s(X)} \leq \lim_n \inf_{Y \subset X, \dim Y = n} \sup_{Z \subset Y} \lambda^+(Z) \quad\mbox{and}\quad 2^{1/s(X)} \leq\lim_n \inf_{Y \subset X, \dim Y = n} \sup_{Z \subset Y} \beta(Z) ,
$$
thus establishing \eqref{eq:inf sup}.

To establish \eqref{eq:convex sup inf}, a similar reasoning is involved; instead of Example \ref{example L_p} one uses the equality $\alpha(L_p) = 2^{1/p}$ from \cite{borwein-sims}.
\end{proof}

Now we characterize the Banach lattices $X$ with $\lambda^+(X)=1$.

\begin{theorem}\label{characterization of banach lattices having l>1}
A Banach lattice $X$ contains lattice-almost isometric copies of $\ell_\infty^2$ if and only if $\lambda^+(X)=1$.
\end{theorem}

\begin{proof}
If $X$ contains lattice-almost isometric copies of $\ell_\infty^2$, by Corollary \ref{+schaefer and subspaces} we obtain $1\leq \lambda^+(X)\leq \lambda^+(\ell_\infty^2)=1$. Conversely, assume that $\lambda^+(X)=1$ and let $c>1$ be given. Our goal is to find a lattice isomorphism $T = T_c : \ell_\infty^2 \to X$ so that $\|T\| \|T^{-1}\| < c$. 

Fix $\varepsilon > 0$ so that $\frac{1+\varepsilon}{1-\varepsilon} < c$. From the definition of $\lambda^+(X)$, there are $x_\varepsilon,y_\varepsilon\in S_X^+$  such that  $1\leq\|x_\varepsilon+y_\varepsilon\|<1+\varepsilon$.
If $a,b\in\mathbb R$, then 
\begin{align*}
    |ax_\varepsilon+by_\varepsilon|\leq|a|x_\varepsilon+|b|y_\varepsilon\leq\max\{|a|,|b|\}(x_\varepsilon+y_\varepsilon).
\end{align*}
So, $\|ax_\varepsilon+by_\varepsilon\|\leq\max\{|a|,|b|\}(1+\varepsilon)$. Also, as $x_\varepsilon,y_\varepsilon\in S_X^+$, we have $\max\{|a|,|b|\}\leq\|ax_\varepsilon+by_\varepsilon\|$ if $a$ and $b$ have the same sign.

Now note that $x_\varepsilon - y_\varepsilon = 2 x_\varepsilon - (x_\varepsilon + y_\varepsilon)$, hence
\begin{equation}
\label{eq:difference}
\|x_\varepsilon - y_\varepsilon\| \geq 2 \|x_\varepsilon\| - \|x_\varepsilon + y_\varepsilon\| \geq 1 - \varepsilon .
\end{equation}
Let $z_\varepsilon = x_\varepsilon \wedge y_\varepsilon$, $x'_\varepsilon = x_\varepsilon - z_\varepsilon$, and $y'_\varepsilon = y_\varepsilon - z_\varepsilon$. Clearly $x'_\varepsilon, y'_\varepsilon$ are positive and disjoint.
Further, $x'_\varepsilon - y'_\varepsilon = x_\varepsilon - y_\varepsilon$, hence $x'_\varepsilon + y'_\varepsilon = |x_\varepsilon - y_\varepsilon|$. By \eqref{eq:difference},
$\|x'_\varepsilon + y'_\varepsilon\| = \|x_\varepsilon - y_\varepsilon\| \geq 1 - \varepsilon$.

We clearly have $x'_\varepsilon \leq x_\varepsilon$, hence $\|x'_\varepsilon\| \leq 1$. In fact, $\|x'_\varepsilon + z_\varepsilon\| = \|x_\varepsilon\| = 1$,
while
$$
\|x'_\varepsilon + 2z_\varepsilon\| \leq \|x'_\varepsilon + y'_\varepsilon + 2z_\varepsilon\| = \|x_\varepsilon + y_\varepsilon\| \leq 1+\varepsilon .
$$
By the triangle inequality,
$$
\|x'_\varepsilon\| \geq 2 \|x'_\varepsilon + z_\varepsilon\| - \|x'_\varepsilon + 2z_\varepsilon\| \geq 1 - \varepsilon ,
$$
and similarly, $\|y'_\varepsilon\| \geq 1 - \varepsilon$. For any $a,b \in \mathbb R$, we have
$$
\|a x'_\varepsilon + b y'_\varepsilon\| \leq \| |a| x_\varepsilon + |b| y_\varepsilon\| \leq (1+\varepsilon) \max \{ |a|, |b| \} ,
$$
and also,
$$
\|a x'_\varepsilon + b y'_\varepsilon\| \geq \max \{ |a| \|x'_\varepsilon\| , |b| \|y'_\varepsilon\| \} \geq (1-\varepsilon) \max \{ |a|, |b| \} ,
$$
Thus, the operator
$$
T : \ell_\infty^2 \to X : (a,b) \mapsto a x'_\varepsilon + b y'_\varepsilon 
$$
has the desired properties.
\end{proof}

In contrast to Example \ref{counterexample}, we obtain: 

\begin{corollary}
    Let $X$ be a Banach lattice. Then $\lambda^+(X)>1$ if and only if $\beta(X)>1$.
\end{corollary}

\begin{proof}
Since $\lambda^+(X)\leq\beta(X)$, $\lambda^+(X)>1$ implies $\beta(X)>1$. Conversely, if $\lambda^+(X)=1$, then, by Theorem \ref{characterization of banach lattices having l>1}, $X$ contains lattice-almost isometric copies of $\ell_\infty^2$, hence, by Corollary \ref{+schaefer and subspaces}, $\beta(X)\leq\beta(\ell_\infty^2)=1$.
\end{proof}

\section{Connection between the lattice Sch\"affer constant and other parameters}
\label{algunas propiedades de lamba^+}

We begin by relating the lattice Sch\"affer constant $\lambda^+(X)$ and the characteristic of monotonicity $\tilde\varepsilon_{0,m}(X).$

\begin{proposition}
    Let $X$ be a Banach lattice. For each $\varepsilon\in[0,1]$ we have
    \begin{gather}\label{relationship between sigma and lambda}
        \lambda^+(X)\leq\sigma_X(\varepsilon)+2-\varepsilon.
    \end{gather}
   Consequently, $\tilde\varepsilon_{0,m}(X)\leq 2-\lambda^+(X).$
\end{proposition}

\begin{proof}
If $\varepsilon\in[0,1]$ is given, by Lemma \ref{continuity of eta} we obtain
\begin{gather*}
        \lambda^+(X)-1-\sigma_X(\varepsilon)=|\sigma_X(1)-\sigma_X(\varepsilon)|\leq 1-\varepsilon.
    \end{gather*}
Thus, $\lambda^+(X)\leq\sigma_X(\varepsilon)+2-\varepsilon.$ The second inequality follows by taking $\varepsilon=\tilde\varepsilon_0(X)$ in \eqref{relationship between sigma and lambda} and recalling that $\sigma_X(\tilde\varepsilon_{0,m}(X))=0$ (Corollary \ref{sigma(e_0)=0}).
\end{proof}

\begin{proposition}\label{tilde e_0<1 equivale a lambda>1}
    Let $X$ be a Banach lattice. Then, $\tilde\varepsilon_0(X)<1$ if and only if $\lambda^+(X)>1$.
\end{proposition}

\begin{proof}
    Suppose that $\tilde\varepsilon_0(X)<1$ and let $\tilde\varepsilon_0(X)<a<1$. Then $\sigma_X(a)>0$. If $x,y\in S_X^+$, we have
    \begin{gather*}
        \|x+y\|\geq\|x+ay\|\geq1+\sigma_X(a).
    \end{gather*}
    Consequently, $\lambda^+(X)\geq1+\sigma_X(a)>1$. Conversely, if $\lambda^+(X)>1$, then $\sigma_X(1)>0$. By Lemma \ref{continuity of eta} there is
    $a<1$ such that $\sigma_X(a)>0$. Thus, $\tilde\varepsilon_0(X)<a<1$.
\end{proof}

\begin{theorem}\label{lambda^+>1 implica que e_0<1}
    If $X$ is a Banach lattice, then 
    \begin{gather*}
        \delta_{m,X}\left(\frac{1}{\lambda^+(X)}\right)=\frac{\lambda^+(X)-1}{\lambda^+(X)}.
    \end{gather*}
    Consequently, $\varepsilon_{0,m}(X)\leq\dfrac{1}{\lambda^+(X)}$. Moreover, $\dfrac{1}{1-\delta_{m,X}(1/2)}\leq\lambda^+(X)$.
\end{theorem}

\begin{proof}
    Plugging $\varepsilon = 1$ into Fact \ref{relationship between delta and eta}, we obtain
    \begin{gather*}
        \delta_{m,X}\left(\frac{1}{\lambda^+(X)}\right)=\frac{\lambda^+(X)-1}{\lambda^+(X)}.
    \end{gather*}
    Now, the first inequality holds when $\lambda^+(X)=1$. Now, if $\lambda^+(X)>1$, the above equation implies that $ \delta_{m,X}\left(\frac{1}{\lambda^+(X)}\right)>0$.
    Whence, $\varepsilon_{0,m}(X)\leq1/\lambda^+(X)$. 

    On the other hand, from Fact \ref{Inequalities between delta and eta} and Lemma \ref{continuity of eta} we obtain
    \begin{gather*}
        \frac{\delta_{m,X}(1/2)}{1-\delta_{m,X}(1/2)}=\lim_{\varepsilon\to1^-}\frac{\delta_{m,X}\left(\dfrac{\varepsilon}{1+\varepsilon}\right)}{1-\delta_{m,X}\left(\dfrac{\varepsilon}{1+\varepsilon}\right)}\leq\lim_{\varepsilon\to1^-}\sigma_X(\varepsilon)=\lambda^+(X)-1.\qedhere
    \end{gather*}
\end{proof}

\begin{remark}
    In \cite{kurc} and \cite[Lemma 2.2]{betiuk-prus} it is claimed that if $X$ is Banach lattice, then
    \begin{gather}\label{formula falsa}
        \delta_{m,X}(\varepsilon)=\frac{\sigma_X(\varepsilon)}{1+\sigma_X(\varepsilon)},\quad\mbox{for all $\varepsilon\in[0,1)$.}
    \end{gather}
    However, this formula is false. Indeed, if \eqref{formula falsa} were true, from \cite[Theorem 2.1]{fora and all} and Fact \ref{continuity of eta} we would have
    \begin{gather*}
        1-\varepsilon_{0,m}(X)=\lim_{\varepsilon\to1^-}\delta_{m,X}(\varepsilon)=\lim_{\varepsilon\to1^-}\frac{\sigma_X(\varepsilon)}{1+\sigma_X(\varepsilon)}=\frac{\lambda^+(X)-1}{\lambda^+(X)}.
    \end{gather*}
    Thus, $\varepsilon_{0,m}(X)=1/\lambda^+(X)$, which is false when $X$ is uniformly monotone.

One can also reach the same conclusion by noting that $\delta_{m,L_1(0,1)}(\varepsilon)=1-\varepsilon=\sigma_{L_1(0,1)}(\varepsilon)$.
\end{remark}

\begin{proposition}\label{e_0<1 es equivalente a lambda^+>1}
     Let $X$ be a Banach lattice. Then $\lambda^+(X)>1$ if and only if $\varepsilon_{0,m}(X)<1$.
\end{proposition}

\begin{proof}
    If $\lambda^+(X)>1$, by Theorem \ref{lambda^+>1 implica que e_0<1} we have $\varepsilon_{0,m}(X)<1$. On the other hand, if $\lambda^+(X)=1$, by Theorem \ref{characterization of banach lattices having l>1}, $X$ contains lattice-almost isometric copies of $\ell_\infty^2$. 
    Fix $\varepsilon > 0$, and find a Banach lattice isomorphism $T = T_\varepsilon\colon\ell_\infty^2\to X$ such that $\|T\| = 1$, and $\|T^{-1}\| < 1+\varepsilon$. Let $x=T(1,1)$ and $y=T(0,1)$, then, ${\bf0}\leq y\leq x$, $\|x\| \leq 1$, and $\|y\| \geq 1-\varepsilon$.
    Therefore, $\delta_{m,X}(1-\varepsilon)\leq 1-\|x-y\|=1-\|T(1,0)\|\leq 1-(1-\varepsilon) = \varepsilon$. Passing to the limit as $\varepsilon \to 1$, we conclude that $\delta_{m,X}(1-)=0$, and so, $\varepsilon_{0,m}(X)=1$.
\end{proof}

To state the next corollary, recall that $X$ is a \textit{KB-space} if every monotone sequence in $B_X$ is norm convergent. It is known that $X$ is a KB-space
if and only if $X$ contains no sublattice isomorphic to $c_0$ \cite[Theorem 2.4.12]{meyer}.

\begin{theorem}\label{Banach lattices with l>1 satisfacen q-estimativa}
     Let $X$ be a Banach lattice. If $\lambda^+(X)>1$, then $X$ satisfies a lower $q$-estimate for some $q\in(1,\infty)$. In particular, if $\lambda^+(X)>1$, then $X$ is a KB-space.  
\end{theorem}
  
\begin{proof}
  If $\varepsilon_{0,m}(X)<1$, by \cite[Theorem 2.3(2)]{betiuk-prus} $X$ satisfies a lower $q$-estimate for some $q\in(1,\infty)$.  
   The conclusion follows from Proposition \ref{e_0<1 es equivalente a lambda^+>1}. For second part, notice that a Banach lattice satisfying a lower $q$-estimate cannot contain a lattice copy of $c_0$, hence it must be a KB-space
   by the aforementioned result. 
\end{proof}

\begin{remark}
    The second part of Theorem \ref{Banach lattices with l>1 satisfacen q-estimativa} follows also from Corollary \ref{+schaefer and subspaces}. Indeed,  If $X$ is not a KB space, then, by \cite[Theorem 2.4.12]{meyer}, it contains a sublattice isomorphic to $c_0$. By \cite[Theorem 2]{chen}, $X$ contains lattice-almost isometric copies of $c_0$.
    Since $\lambda^+(c_0)=1$, we obtain $\lambda^+(X)=1$ by Corollary \ref{+schaefer and subspaces}. 
    
    On the other hand, being a KB-space does not imply that $\lambda^+>1$ as the following example shows. Let $(p_n)$ be a sequence in $(1,\infty)$ such that $\lim_{n\to\infty} p_n=\infty$ and consider the Banach lattice $X=(\bigoplus_{n\in\mathbb N}\ell_{p_n}^2)_2$. Since $X$ is reflexive, $X$ is a KB-space. On the other hand, $\lambda^+(X)\leq 2^{1/p_n}$ for all $n\in\mathbb N$.
    Hence, $\lambda^+(X)=1$.
\end{remark}

Now we give a condition to obtain superreflexivity of $X$ in terms of the characteristic $\tilde\varepsilon_{0,m}(X)$. 

\begin{corollary}
    Let $X$ be a Banach lattice. If $\tilde\varepsilon_{0,m}(X)+\alpha(X)<2$, then $X$ is superreflexive.
\end{corollary}

\begin{proof}
    Since $1\leq\alpha(X)\leq2$, we obtain $\tilde\varepsilon_{0,m}(X)<2-\alpha(X)\leq 1$. From Proposition \ref{tilde e_0<1 equivale a lambda>1}
    it follows that $\lambda^+(X)>1.$ By Corollary \ref{Banach lattices with l>1 satisfacen q-estimativa}, $X$ satisfies a lower $q$-estimate for some $q\in(1,\infty)$.  
    On the other hand, since $\alpha(X)<2$, $X$ satisfies an upper $p$-estimate for some $p\in(1,\infty)$ by \cite[Theorem 2.3]{betiuk-prus}.
    Thus, $X$ is superreflexive by \cite[Theorem 1.f.10]{lindenstrauss-tzafriri}.
\end{proof}

 In \cite[Theorem 3.2]{prus} it is proved that if $2\varepsilon_{0,m}(X)+\alpha(X)<2$, then $X$ is uniformly non-square. Notice that if $2\varepsilon_{0,m}(X)+\alpha(X)<2$, by \eqref{inequalities envolving e_0 and tilde e_0} we also have $\tilde\varepsilon_{0,m}(X)+\alpha(X)<2$. This leads to the natural question:

\begin{question}
    Let $X$ be a Banach lattice. Is $X$ uniformly non-square provided that $\tilde\varepsilon_{0,m}(X)+\alpha(X)<2$?
\end{question}

\section*{Acknowledgments}

The first author thanks to Universidad Industrial de Santander.
 
\vspace{2mm}

\section*{Declarations}

\textbf{Ethical Approval} 

Not applicable.

\

\textbf{Funding}

First author: Universidad Industrial de Santander.

\

\textbf{Availability of data and materials} 

Not applicable.

\

\textbf{Financial interests}

The authors have no relevant financial or non-financial interests to disclose.

\

\textbf{Conflict of Interests/Competing Interests}

The authors declare that they have no conflicts of interest/competing interests.

\bibliographystyle{amsalpha}

\end{document}